\documentclass[10pt]{article}
\usepackage{fancyhdr,a4wide}
\usepackage[mathscr]{eucal}
\usepackage{amsmath}
\usepackage{amssymb}
\usepackage{amsthm}
\usepackage{float}
\usepackage{caption}
\usepackage[caption = false]{subfig}
\usepackage{enumerate}
\usepackage{fullpage}
\usepackage{graphicx}
\usepackage{multicol}
\usepackage{tikz,  mathtools, soul, mathdots, yhmath}
\usetikzlibrary{calc,decorations.markings,arrows}

\usepackage{hyperref}

\input xy
\xyoption{all}

\usepackage[all]{xy}

\usepackage{color}

\title{Indecomposable Modules in the Grassmannian Cluster Category ${\rm CM}(B_{5,10})$} 
\author{Dusko Bogdanic and Ivan-Vanja Boroja}
\date{} 

\usepackage{pst-all}  

\definecolor{light-grey}{gray}{0.6}  

\begin{document}

\newtheorem{lm}{Lemma}[section]
\newtheorem{prop}[lm]{Proposition}
\newtheorem{satz}[lm]{Satz}

\newtheorem{corollary}[lm]{Corollary}
\newtheorem{cor}[lm]{Korollar}
\newtheorem{claim}[lm]{Claim}
\newtheorem{theorem}[lm]{Theorem}
\newtheorem*{thm}{Theorem}

\theoremstyle{definition}
\newtheorem{defn}[lm]{Definition}
\newtheorem{qu}{Question}
\newtheorem{ex}[lm]{Example}
\newtheorem{exas}[lm]{Examples}
\newtheorem{exc}[lm]{Exercise}
\newtheorem*{facts}{Facts}
\newtheorem{rem}[lm]{Remark}

\theoremstyle{remark}

\newcommand{\perm}{\operatorname{Perm}\nolimits}
\newcommand{\NN}{\operatorname{\mathbb{N}}\nolimits}
\newcommand{\ZZ}{\operatorname{\mathbb{Z}}\nolimits}
\newcommand{\QQ}{\operatorname{\mathbb{Q}}\nolimits}
\newcommand{\RR}{\operatorname{\mathbb{R}}\nolimits}
\newcommand{\CC}{\operatorname{\mathbb{C}}\nolimits}
\newcommand{\PP}{\operatorname{\mathbb{P}}\nolimits}
\newcommand{\cA}{\operatorname{\mathcal{A}}\nolimits}
\newcommand{\LL}{\operatorname{\Lambda}\nolimits}

\newcommand{\MM}{\operatorname{\mathbb{M}}\nolimits}
\newcommand{\Fk}{\mathcal{F}_{k,n}}
\newcommand{\Mk}{M_{k,n}}

\newcommand{\ad}{\operatorname{ad}\nolimits}
\newcommand{\im}{\operatorname{im}\nolimits}
\newcommand{\Char}{\operatorname{char}\nolimits}
\newcommand{\Aut}{\operatorname{Aut}\nolimits}

\newcommand{\id}{\operatorname{id}\nolimits}
\newcommand{\Id}{\operatorname{Id}\nolimits}
\newcommand{\ord}{\operatorname{ord}\nolimits}
\newcommand{\ggT}{\operatorname{ggT}\nolimits}
\newcommand{\lcm}{\operatorname{lcm}\nolimits}

\newcommand{\Gr}{\operatorname{Gr}\nolimits}

\maketitle

\begin{abstract} 
In this paper we study indecomposable rank 2 modules in the Grassmannian cluster category ${\rm CM}(B_{5,10})$. This is the smallest wild case containing modules whose profile layers are $5$-interlacing. We construct all rank 2 indecomposable modules with filtration $\{i,i+2,i+4,i+6,i+8\}\mid \{i+1,i+3,i+5,i+7,i+9\}$, classify them up  to isomorphism,  and  parameterize all infinite  families of non-isomorphic rank 2 modules.
\end{abstract}

\noindent
%
%
\section{Introduction and preliminaries}

In  their seminal work \cite{FZ}, Fomin and Zelevinsky  used the homogeneous coordinate ring $\CC[\Gr(2,n)]$ of the Grassmannian of $2$-dimensional subspaces of $\CC^n$ as one of the first examples of the theory of cluster algebras. Scott proved in~\cite{Scott06} that this cluster structure can be generalized to the coordinate ring $\CC[\Gr(k,n)]$. These results initiated a lot of research activities in cluster theory, e.g.  \cite{SW, gls, HL10, gssv, MuS, BKM16, JKS16, mr,fraser}.  Geiss, Leclerc, and Schroer \cite{rigid, GLS08} gave an aditive categorification of the cluster algebra structure on the homogeneous coordinate ring  of the Grassmannian variety of $k$-dimensional subspaces in $\mathbb C^n$ in terms of a subcategory of the category of finite dimensional modules over the preprojective algebra of type $A_{n-1}$. Jensen, King, and Su \cite{JKS16} introduced a new additive categorification of this cluster structure using the maximal Cohen-Macaulay modules over the completion of an algebra $B_{k,n}$ which is a quotient of the preprojective algebra of type $A_{n-1}$.  In the category ${\rm CM}(B_{k,n}) $ of Cohen-Macaulay modules over $B_{k,n}$, among the indecomposable modules are the rank $1$ modules which are known to be in bijection with $k$-subsets of $\{1,2,\dots,n\}$, and their explicit construction has been given in \cite{JKS16} (a $k$-subset $I$ corresponds to the rank 1 module $L_I$). Rank 1 modules are the building blocks of the category as any module in ${\rm CM}(B_{k,n}) $ can be filtered by rank $1$ modules (the filtration is noted in the profile of a module, \cite[Corollary 6.7]{JKS16}). The number of rank 1 modules appearing in the filtration of a given module is called the rank of that module. 

The aim of this paper is to explicitly construct rank 2 indecomposable Cohen-Macaulay $B_{k,n}$-modules in the case when $k=5$ and $n=10$. All indecomposable modules of rank 2 whose rank 1 filtration layers $L_I$ and $L_J$ satisfy the condition $|I\cap J|\geq k-4$ have been constructed in \cite{BBL}. This covers all tame cases and the wild case $(4,9)$. The case $(5,10)$ is the smallest wild case that contains rank 2 indecomposable modules whose layers are $5$-interlacing. In this case, the only profiles with $5$-interlacing layers are of the form $\{i,i+2,i+4,i+6,i+8\}\mid \{i+1,i+3,i+5,i+7,i+9\}$, where $i=1,2$.  We construct all indecomposable modules with filtration $\{i,i+2,i+4,i+6,i+8\}\mid \{i+1,i+3,i+5,i+7,i+9\}$, classify them up  to isomorphism,  and  parameterize all infinite  families of non-isomorphic rank 2 modules. It is important to remark that even though we only treat the case $(5,10)$ in this paper, all arguments and results are also valid for the general case $(k,n)$ for all rank 2 modules with  tightly $5$-interlacing layers. 

We follow the exposition from \cite{JKS16, BB, BBGE} in order to introduce notation and background results. Here, the central combinatorial notion is that of $r$-interlacing.

\begin{defn}[$r$-interlacing] \label{interlacing}
Let $I$ and $J$ be two $k$-subsets of $\{1,\dots,n\}$. The sets $I$ and $J$ are said to be {\em $r$-interlacing} if there exist subsets  $\{i_1,i_3,\dots,i_{2r-1}\}\subset I\setminus J$ and $\{i_2,i_4,\dots, i_{2r}\}\subset J\setminus I$ 
such that $i_1<i_2<i_3<\dots <i_{2r}<i_1$ (cyclically) 
and if there exist no larger subsets of $I$ and of $J$ with this property.  We say that $I$ and $J$ are 
{\em tightly $r$-interlacing} if they are $r$-interlacing and $|I\cap J|=k-r.$
\end{defn}

 Let $\Gamma_n$ be the quiver of the boundary algebra, with vertices $1,2,\dots, n$ 
on a cycle and arrows $x_i: i-1\to i$, $y_i:i\to i-1$. 
We write ${\rm CM}(B_{k,n})$ for the category of maximal Cohen-Macaulay modules for  
the completed path algebra $B_{k,n}$ of $\Gamma_n$, with relations 
$xy-yx$ and $x^k-y^{n-k}$ (at every vertex). The centre of $B_{k,n}$ is 
$Z:=\CC[|t|]$, where $t=\sum_ix_iy_i$.  For example,  in the following figure we have the quiver $\Gamma_n$ for $n=5.$ 
\begin{figure}[H]
\begin{center}
\begin{tikzpicture}[scale=1]
\newcommand{\radius}{1.5cm}
\foreach \j in {1,...,5}{
  \path (90-72*\j:\radius) node[black] (w\j) {$\bullet$};
  \path (162-72*\j:\radius) node[black] (v\j) {};
  \path[->,>=latex] (v\j) edge[black,bend left=30,thick] node[black,auto] {$x_{\j}$} (w\j);
  \path[->,>=latex] (w\j) edge[black,bend left=30,thick] node[black,auto] {$y_{\j}$}(v\j);
}
\draw (90:\radius) node[above=3pt] {$5$};
\draw (162:\radius) node[above left] {$4$};
\draw (234:\radius) node[below left] {$3$};
\draw (306:\radius) node[below right] {$2$};
\draw (18:\radius) node[above right] {$1$};
\end{tikzpicture}
\end{center}
\caption{The quiver $\Gamma_5$} \label{quiver5}
\end{figure}
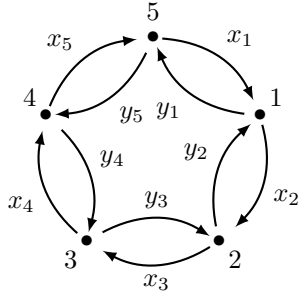

The algebra $B_{k,n}$  coincides with the quotient of the completed path
algebra of the graph $C$ (a circular graph with vertices
$C_0=\mathbb Z_n$ set clockwise around a circle, and with the set of edges, $C_1$, also
labeled by $\mathbb Z_n$, with edge $i$ joining vertices $i-1$ and $i$), i.e., the doubled quiver as above,
by the closure of the ideal generated by the relations above (we view the completed path
algebra of the graph $C$ 
as a topological algebra via the $m$-adic topology, where $m$ is the two-sided ideal
generated by the arrows of the quiver, see \cite[Section 1]{DWZ08}). The algebra 
$B_{k,n}$ was introduced in \cite[Section 3]{JKS16}. 
Observe that $B_{k,n}$ is isomorphic to $B_{n-k,n}$, so we will always take $k\le \frac n 2$. Moreover, throughout this paper, we will be working with the algebra $B_{5,10}$. For background results on algebras given by quivers and relations and their representations we recommend \cite{ASS}.

The (maximal) Cohen-Macaulay $B_{k,n}$-modules are precisely those which are
free as $Z$-modules. Such a module $M$ is given by a representation
$\{M_i\,:\,i\in C_0\}$ of
the quiver with each $M_i$ a free $Z$-module of the same rank
(which is the rank of $M$).

\begin{defn}[\cite{JKS16}, Definition 3.5]
For any $B_{k,n}$-module $M$ and $K$ the field of fractions of $Z$, the {\bf rank} 
of $M$, denoted by ${\rm rk}(M)$,  is defined 
to be ${\rm rk}(M) = {\rm len}(M \otimes_Z K)$. 
\end{defn}

Note that $B_{k,n}\otimes_Z K\cong M_n ( K)$, 
which is a simple algebra. It is easy to check that the rank is additive on short exact sequences,
that ${\rm rk} (M) = 0$ for any finite-dimensional $B_{k,n}$-module 
(because these are torsion over $Z$) and 
that, for any Cohen-Macaulay $B_{k,n}$-module $M$ and every idempotent $e_j$, $1\leq j\leq n$, ${\rm rk}_Z(e_j M) = {\rm rk}(M)$, so that, in particular, ${\rm rk}_Z(M) = n  {\rm rk}(M)$.

\begin{defn}[\cite{JKS16}, Definition 5.1] \label{d:moduleMI}
For any $k$-subset   $I$  of $C_1$, we define a rank $1$ $B_{k,n}$-module
\[
  L_I = (U_i,\ i\in C_0 \,;\, x_i,y_i,\, i\in C_1)
\]
as follows.
For each vertex $i\in C_0$, set $U_i=\mathbb C[[t]]$ and,
for each edge $i\in C_1$, set
\begin{itemize}
\item[] $x_i\colon U_{i-1}\to U_{i}$ to be multiplication by $1$ if $i\in I$, and by $t$ if $i\not\in I$,
\item[] $y_i\colon U_{i}\to U_{i-1}$ to be multiplication by $t$ if $i\in I$, and by $1$ if $i\not\in I$.
\end{itemize}
\end{defn}

The module $L_I$ can be represented by a lattice diagram
$\mathcal{L}_I$ in which $U_0,U_1,U_2,\ldots, U_n$ are represented by columns of vertices (dots) from
left to right (with $U_0$ and $U_n$ to be identified), going down infinitely. 
The vertices in each column correspond to the natural monomial 
$\mathbb C$-basis of $\mathbb C[t]$.
The column corresponding to $U_{i+1}$ is displaced half a step vertically
downwards (respectively, upwards) in relation to $U_i$ if $i+1\in I$
(respectively, $i+1\not \in I$), and the actions of $x_i$ and $y_i$ are
shown as diagonal arrows. Note that the $k$-subset $I$ can then be read off as
the set of labels on the arrows pointing down to the right which are exposed
to the top of the diagram. For example, the lattice diagram $\mathcal{L}_{\{1,4,5\}}$
in the case $k=3$, $n=8$, is shown in the following picture   
\begin{center}
\begin{figure}[h]
\center
\begin{tikzpicture}[scale=0.8,baseline=(bb.base),
quivarrow/.style={black, -latex, thin}]
\newcommand{\seventh}{51.4} 
\newcommand{\circradius}{1.5cm}
\newcommand{\inradius}{1.2cm}
\newcommand{\outradius}{1.8cm}
\newcommand{\dotrad}{0.1cm} 
\newcommand{\bdrydotrad}{{0.8*\dotrad}} 
\path (0,0) node (bb) {}; 


\draw (0,0) circle(\bdrydotrad) [fill=black];
\draw (0,2) circle(\bdrydotrad) [fill=black];
\draw (1,1) circle(\bdrydotrad) [fill=black];
\draw (2,0) circle(\bdrydotrad) [fill=black];
\draw (2,2) circle(\bdrydotrad) [fill=black];
\draw (3,1) circle(\bdrydotrad) [fill=black];
\draw (3,3) circle(\bdrydotrad) [fill=black];
\draw (4,0) circle(\bdrydotrad) [fill=black];
\draw (4,2) circle(\bdrydotrad) [fill=black];
\draw (5,1) circle(\bdrydotrad) [fill=black];
\draw (6,0) circle(\bdrydotrad) [fill=black];
\draw (6,2) circle(\bdrydotrad) [fill=black];
\draw (7,1) circle(\bdrydotrad) [fill=black];
\draw (7,3) circle(\bdrydotrad) [fill=black];
\draw (8,2) circle(\bdrydotrad) [fill=black];
\draw (8,4) circle(\bdrydotrad) [fill=black];
\draw (8,0) circle(\bdrydotrad) [fill=black];


\draw [quivarrow,shorten <=5pt, shorten >=5pt, ultra thick] (0,2)-- node[above]{$1$} (1,1);
\draw [quivarrow,shorten <=5pt, shorten >=5pt] (1,1) -- node[above]{$1$} (0,0);
\draw [quivarrow,shorten <=5pt, shorten >=5pt, ultra thick] (2,2) -- node[above]{$2$} (1,1);
\draw [quivarrow,shorten <=5pt, shorten >=5pt] (1,1) -- node[above]{$2$} (2,0);
\draw [quivarrow,shorten <=5pt, shorten >=5pt, ultra thick] (3,3) -- node[above]{$3$} (2,2);
\draw [quivarrow,shorten <=5pt, shorten >=5pt] (2,2) -- node[above]{$3$} (3,1);
\draw [quivarrow,shorten <=5pt, shorten >=5pt] (3,1) -- node[above]{$3$} (2,0);
\draw [quivarrow,shorten <=5pt, shorten >=5pt, ultra thick] (3,3) -- node[above]{$4$} (4,2);
\draw [quivarrow,shorten <=5pt, shorten >=5pt] (4,2) -- node[above]{$4$} (3,1);
\draw [quivarrow,shorten <=5pt, shorten >=5pt] (3,1) -- node[above]{$4$} (4,0);
\draw [quivarrow,shorten <=5pt, shorten >=5pt, ultra thick] (4,2) -- node[above]{$5$} (5,1);
\draw [quivarrow,shorten <=5pt, shorten >=5pt] (5,1) -- node[above]{$5$} (4,0);
\draw [quivarrow,shorten <=5pt, shorten >=5pt, ultra thick] (6,2) -- node[above]{$6$} (5,1);
\draw [quivarrow,shorten <=5pt, shorten >=5pt] (5,1) -- node[above]{$6$} (6,0);
\draw [quivarrow,shorten <=5pt, shorten >=5pt] (6,2) -- node[above]{$7$} (7,1);
\draw [quivarrow,shorten <=5pt, shorten >=5pt] (7,1) -- node[above]{$7$} (6,0);
\draw [quivarrow,shorten <=5pt, shorten >=5pt, ultra thick] (7,3) -- node[above]{$7$} (6,2);
\draw [quivarrow,shorten <=5pt, shorten >=5pt] (7,3) -- node[above]{$8$} (8,2);
\draw [quivarrow,shorten <=5pt, shorten >=5pt] (8,2) -- node[above]{$8$} (7,1);
\draw [quivarrow,shorten <=5pt, shorten >=5pt, ultra thick] (8,4) -- node[above]{$8$} (7,3);
\draw [quivarrow,shorten <=5pt, shorten >=5pt] (7,1) -- node[above]{$8$} (8,0);

\draw [dotted] (0,-2) -- (0,2);
\draw [dotted] (8,-2) -- (8,4);

\draw [dashed] (4,-2) -- (4,-1);
\end{tikzpicture}
\caption{Lattice diagram of the module $L_{\{1,4,5\}}$} \label{Lattice}
\end{figure}
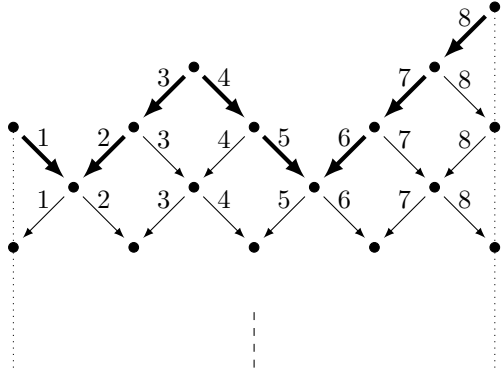
\end{center}

We see from Figure \ref{Lattice} that the module $L_I$ is determined by its upper boundary, denoted  by the thick lines, 
which we refer to as the {\em rim} of the module $L_I$ (this is why we call the $k$-subset $I$  the rim of $L_I$). 
Throughout this paper we will identify a rank 1 module $L_I$ with its rim. Moreover, most of the time we will omit the arrows in the rim of $L_I$ and represent it as an undirected graph.

\begin{prop}[\cite{JKS16}, Proposition 5.2]
Every rank $1$ Cohen-Macaulay $B_{k,n}$-module is isomorphic to $L_I$
for some unique $k$-subset $I$ of $C_1$.
\end{prop}

Every $B_{k,n}$-module has a canonical endomorphism given by multiplication by $t\in Z$.
For ${L}_I$ this corresponds to shifting $\mathcal{L}_I$ one step downwards.
Since $Z$ is central, ${\rm Hom}_{B_{k,n}}(M,N)$ is
a $Z$-module for arbitrary $B_{k,n}$-modules $M$ and $N$. If $M,N$ are free $Z$-modules, then so is ${\rm Hom}_{B_{k,n}}(M,N)$. In particular, for any two rank 1 
Cohen-Macaulay $B_{k,n}$-modules $L_I$ and $L_J$, ${\rm Hom}_{B_{k,n}}(L_I,L_J)$ is a free 
module of rank 1
over $Z=\mathbb C[[t]]$, generated by the canonical map given by placing the 
lattice of $L_I$ inside the lattice of $L_J$ as far up as possible so that no part of the rim of $L_I$ is strictly above the rim of $L_J$ \cite[Section 6]{JKS16}.

Every indecomposable module $M$ of rank $n$ in ${\rm CM}(B_{k,n})$ has a filtration having factors 
$L_{I_1},L_{I_{2}},\dots, L_{I_n}$ of rank 1. 
This filtration is noted in its \emph{profile}, 
${\rm pr} (M) = I_1 \mid I_2\mid\ldots \mid I_n$, \cite[Corollary 6.7]{JKS16}.
In the case of  a rank $2$ module $M$ with filtration $L_I\mid L_J$ (i.e., with profile $I\mid J$), 
we picture the profile of this module by drawing the rim $J$ below the rim $I$, in such a way that $J$ is placed as far up as possible so that no part of the rim $J$ is strictly above the rim $I$. 
Note that there is at least one point where the rims $I$ and $J$ meet (see {Figure~\ref{fig:profil} for an example}). 

\begin{figure}[H]
\begin{center}
\includegraphics[width=8cm]{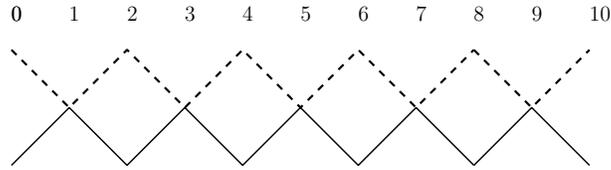}
\caption{The profile $\{1,3,5,7,9\}\mid \{2,4,6,8,10\}$ in ${\rm CM}(B_{5,10})$.}
\label{fig:profil}
\end{center}
\end{figure}

For background on the poset associated to an indecomposable module or to its profile, 
we refer to~\cite[Section 6]{JKS16} and to~\cite[Section 2]{BBGEL20}.

\section{Tight $5$-interlacing}\label{sec:tight-5}
In this section we construct all rank 2 indecomposable modules with the profile $I\mid J$ in the case when $I$ and $J$ are  tightly $5$-interlacing $5$-subsets, i.e., when $|I\setminus J|=|J\setminus I|=5$ 
and non-common elements of $I$ and $J$ interlace, that is, $|I\cap J|=0$. Rank 2 indecomposable modules with $3$-interlacing and $4$-interlacing layers have been constructed and parameterized in \cite{BBL}. 

In the case $(5,10)$, there are only two profiles with  $5$-interlacing layers, namely $I\mid J$ and $J\mid I$, where $I=\{1,3,5,7, 9\}$ and $J=\{2,4,6,8, 10\}$.  We will work with the profile  $I\mid J$, the arguments are the same for $J\mid I.$

In \cite{BBL},  we defined a rank 2 module $\MM(I,J)$ with filtration $L_I\mid L_J$ in a similar way as rank 1 modules 
are defined in ${\rm CM}(B_{k,n})$. We recall the construction here.   
Let $V_i:=\CC[|t|]\oplus\CC[|t|]$, $i=1,\dots, n$. The module $\MM(I,J)$ has $V_i$ at each vertex $1,2,\dots, n$ of $\Gamma_n$.   In order to have a 
module structure for $B_{k,n}$, for every $i$ we need to define  $x_i\colon V_{i-1}\to V_{i}$ and $y_i\colon V_{i}\to V_{i-1}$ in such a way that $x_iy_i=t\cdot \id$ and $x^k=y^{n-k}$.   

Define 
\begin{align*}
&&&&&&x_{2i+1}&=\begin{pmatrix} t& b_{2i+1} \\ 0 & 1 \end{pmatrix},& x_{2i}&=\begin{pmatrix} 1& b_{2i} \\ 0 & t \end{pmatrix}, &&&&&\\ 
&&&&&&y_{2i+1}&=\begin{pmatrix} 1& -b_{2i+1} \\ 0 & t \end{pmatrix}, &y_{2i}&=\begin{pmatrix} t& -b_{2i} \\ 0 & 1 \end{pmatrix},&&&&&
\end{align*}
for $i=0,1,2,3,4$. Also, we assume that $\sum_{i=0}^9b_i=0$. By construction it holds that  $xy=yx$ and $x^5=y^{10-5}$ at all vertices and that $\MM(I,J)$ is free over the centre of $B_{5,10}$. Hence,  $\MM(I,J)$ is in ${\rm CM}(B_{5,10})$.  We say  that the module $\MM(I,J)$ is defined by the tuple $(b_i)$.

It was shown in \cite{BBL} that $\mathbb M(I,J)$ is isomorphic to $L_I\oplus L_J$ if and only if  $t\mid b_{i}+b_{i+1}$, for $i$ odd.

Our aim is to study the structure of the module $\MM(I,J)$ in terms of the divisibility conditions the coefficients $b_i$ satisfy. Since $I$ and $J$ are fixed, $\MM(I,J)$ will be denoted by $\MM$.
 
We distinguish between different cases depending on whether the sums $b_1+b_2$, $b_3+b_4$, $b_5+b_6$, $b_7+b_8$, and $b_9+b_{10}$ are divisible by $t$ or not. We will call these the five {\em divisibility conditions} $t\mid b_1+b_2$, $t\mid b_3+b_4$, 
$t\mid b_5+b_6$, $t\mid b_7+b_8$, and $t\mid b_9+b_{10}$, 
and write {\rm (div)} to abbreviate. Also, we write $B_i=b_i+b_{i+1}$ for odd $i$. 
There are four base cases: one of the sums $B_i$ is divisible by $t$ and four are not, 
two are divisible by $t$ and three are not, three are divisible by $t$ and two are not, and none of the sums is divisible by $t$. Note that it is not possible that four of the sums are divisible by $t$ and one is not because they sum up to 0. 

\begin{theorem}[\cite{BBo}, Theorem 3.3]\label{t6}
The module $\MM(I,J)$ is indecomposable if and only if there exist odd  indices  $i_{l_1}$  and  $i_{l_2}$ such that  $t\mid b_i+b_{i+1}$, for $i_{l_1}<i<i_{l_2}$, $i$ odd, $t\nmid b_{i_{l_1}}+b_{i_{l_1}+1}$, $t\nmid b_{i_{l_{2}}}+b_{i_{l_{2}}+1}$, and $t\nmid b_{i_{l_1}}+b_{i_{l_1}+1}+b_{i_{l_{2}}}+b_{i_{l_{2}}+1}$. 
\end{theorem}

Throughout the paper, in all the cases we consider, we will assume that the assumptions of the previous theorem are fulfilled, i.e., that there are odd  indices  $i_{l_1}$  and  $i_{l_2}$ such that  $t\mid b_i+b_{i+1}$, for $i_{l_1}<i<i_{l_2}$, $i$ odd, $t\nmid b_{i_{l_1}}+b_{i_{l_1}+1}$, $t\nmid b_{i_{l_{2}}}+b_{i_{l_{2}}+1}$, and $t\nmid b_{i_{l_1}}+b_{i_{l_1}+1}+b_{i_{l_{2}}}+b_{i_{l_{2}}+1}$.  This means that one of the base cases, the case where two of the sums $B_i$ are not divisible by $t$ and three are divisible by $t$, will not be considered, because in this case the assumptions of the previous theorem are not fulfilled. More precisely, the sum of the only two $B_i$'s that are not divisible by $t$ is divisible by $t$, because $\sum_{i=1}^{10}b_i=0$. Therefore, there are only three base cases to consider.

We will show that there are infinitely many non-isomorphic modules with the same filtration for the cases when none of the sums is divisible by $t$ and when four of the sums are not divisible by $t$.

Let  $(c_1, c_2, c_3, c_4, c_5, c_6,c_7,c_8 , c_9, c_{10})$ be another $10$-tuple such that $\sum_{i=1}^{10}c_i=0$ and that the module defined by this tuple is indecomposable. Denote this module by $\MM'$ and by $C_i$ the sum $c_i+c_{i+1}$, for odd $i$. We say that the modules $\MM$ and $\MM'$ satisfy the same divisibility conditions if the following holds: $t\mid B_i$ if and only if $t\mid C_i$, and $t\mid B_{i}+B_{j}$ if and only if  $t\mid C_{i}+C_{j}$. 

For the rest of the paper, if  $t^dv=w$, for a positive integer $d$, then $t^{-d}w$ denotes $v$. If there is an isomorphism $\varphi=(\varphi_i)$ between the modules $\MM$ and $\MM'$, then the following holds. 

Let us assume that $\varphi_0=\begin{pmatrix}\alpha & \beta \\ \gamma & \delta  \end{pmatrix}$. Then from $\varphi_ix_i=x_i\varphi_{i-1}$ we get
\begin{align*}
\varphi_{2i+1}&= \begin{pmatrix} \alpha +(c_1+\dots +c_{2i+1})t^{-1}\gamma & \beta t-\alpha \sum_{j=1}^{2i+1} b_j+\delta \sum_{j=1}^{2i+1} c_j  -(\sum_{j=1}^{2i+1} b_j)(\sum_{j=1}^{2i+1} c_j)t^{-1}\gamma  \\ t^{-1}\gamma  & \delta -(b_1+\dots b_{2i+1})t^{-1}\gamma_0 \end{pmatrix} ,\\
\varphi_{2i}&=\begin{pmatrix} \alpha +(c_1+\dots +c_{2i})t^{-1}\gamma & \beta +t^{-1}(-\alpha \sum_{j=1}^{2i} b_j+\delta \sum_{j=1}^{2i} c_j -t^{-1}\gamma\sum_{j=1}^{2i} b_j \sum_{j=1}^{2i} c_j) \\  \gamma & \delta -(b_1+\dots +b_{2i})t^{-1}\gamma \end{pmatrix}, 
\end{align*}
where $t\mid \gamma$ and 
\begin{align}\label{izo} \nonumber
t&\mid -\alpha(b_1+b_2)+\delta (c_1+c_2) -(b_1+b_2)(c_1+c_2)t^{-1}\gamma,\\  \nonumber 
t&\mid -\alpha (b_1+b_2+b_3+b_4)+\delta(c_1+c_2+c_3+c_4)  -(b_1+b_2+b_3+b_4)(c_1+c_2+c_3+c_4)t^{-1}\gamma ,\\ 
t&\mid -\alpha (b_1+b_2+b_3+b_4+b_5+b_6)+\delta(c_1+c_2+c_3+c_4+c_5+c_6) -t^{-1}\gamma \sum_{i=1}^6b_i \sum_{i=1}^6c_i,\\ \nonumber 
t&\mid -\alpha \sum_{i=1}^8b_i+\delta\sum_{i=1}^8b_i -t^{-1}\gamma \sum_{i=1}^8b_i \sum_{i=1}^8c_i.
\end{align}

Since $t\mid \gamma$ and we would like $\varphi$ to be invertible, then it must be that $t\nmid \alpha$ and $t\nmid \delta$. Then the inverse of $\varphi_0$ is $\frac{1}{\alpha \delta -\beta \gamma }\begin{pmatrix} \delta &-\beta  \\  -\gamma  & \alpha  \end{pmatrix}.$ Thus, in order to construct an isomorphism between $\MM$ and $\MM'$, we have to make sure that the divisibility conditions (\ref{izo}) are met for the coefficients of $\varphi_0$. This will be used repeatedly throughout the paper. 

Before considering base cases, in the next theorem we show that if the modules $\MM$ and $\MM'$ do not satisfy the same divisibility conditions, then they are not isomorphic. 

\begin{theorem} \label{izom}
The above defined modules  $\MM$ and $\MM'$ are not isomorphic if they do not satisfy the same divisibility conditions.
\end{theorem}
\begin{proof}
Let us assume that there is an odd index, say $i_1$, such that  $t\nmid b_{i_1}+b_{i_1+1}$ and $t\mid c_{i_1}+c_{i_1+1}$. 
If $\varphi=(\varphi_i)$ is an isomorphism between $\MM$ and $\MM'$, let $\varphi_{i_1-1}=\begin{pmatrix}\alpha & \beta \\ \gamma & \delta  \end{pmatrix}$.  Then the coefficients of $\varphi_{i_1-1}$ have to satisfy divisibility conditions (\ref{izo}). Here, we assume that the cyclic numbering of vertices started at $i_1-1$ instead at $0$. Since $t\mid c_{i_1}+c_{i_1+1}$, the first condition from  (\ref{izo}),    $t\mid -\alpha(b_{i_1}+b_{i_1+1})+\delta (c_{i_1}+c_{i_1+1}) -(b_{i_1}+b_{i_1+1})(c_{i_1}+c_{i_1+1})t^{-1}\gamma,$  reduces to $t\mid \alpha (b_{i_1}+b_{i_1+1})$. But $t\nmid \alpha$ and $t\nmid b_{i_1}+b_{i_1+1}$ which is a contradiction. Hence, $\MM$ and $\MM'$ are not isomorphic in this case. 

Assume that, for every odd $i$, $t\nmid b_{i}+b_{i+1}$ if and only if $t\nmid c_{i}+c_{i+1}$.  Since $\MM$ and $\MM'$ do not satisfy the same divisibility conditions, there is an index, say $i_1$,  such that $t\nmid B_{i_1}+B_{i_2}$ and $t\mid C_{i_1}+C_{i_2}$.   Assuming again that the cyclic numbering of vertices started at $i_1-1$ instead at $0$, then the second divisibility condition from (\ref{izo}), $t\mid -\alpha (B_{i_1}+B_{i_2})+\delta(C_{i_1}+C_{i_2})  -(B_{i_1}+B_{i_2})(C_{i_1}+C_{i_2})t^{-1}\gamma$, reduces to  $t\mid -\alpha (B_{i_1}+B_{i_2})$. But, $t\nmid -\alpha$ and  $t\nmid B_{i_1}+B_{i_2}$ which is a contradiction. Hence, $\MM$ and $\MM'$ are not isomorphic in this case as well. 
\end{proof}

For the remainder of the paper, when we investigate if the modules $\MM$ and $\MM'$ are isomorphic, we will implicitly assume that they satisfy the same divisibility conditions.

\subsection{Three of the sums $B_i$ are not divisible by $t$} 
Assume that  $t\nmid b_{i_l}+b_{i_l+1}$, $l=1,2,3$, and  $t\mid b_{i_l}+b_{i_l+1}$, $l=4,5$, where $\{i_1, i_2, i_3, i_4,i_5\}=\{1,3,5,7,9\}$.  Since $\sum_{i=1}^{10}b_i=0$, it follows that $t\nmid B_{i_l}+B_{i_s}$, for all $l,s\leq 3$. By Theorem \ref{t6}, the constructed module is indecomposable. Denote this module by $\MM_{i_1,i_2,i_3}$. 
Let $(c_i)$ be another tuple giving rise to the module $\MM_{j_1,j_2,j_3}$. The following theorem says that the modules  $\MM_{i_1,i_2,i_3}$ and $\MM_{j_1,j_2,j_3}$ are isomorphic if and only if they satisfy the same divisibility conditions.

\begin{theorem} \label{M123}
The modules  $\MM_{i_1,i_2,i_3}$ and $\MM_{j_1,j_2,j_3}$ are isomorphic if and only if $\{i_1,i_2,i_3\}=\{j_1,j_2,j_3\}.$
\end{theorem}
\begin{proof} 
Let $\{i_1,i_2,i_3\}=\{j_1,j_2,j_3\}$ and $\varphi_{i_1-1}=\begin{pmatrix}\alpha & \beta \\ \gamma & \delta  \end{pmatrix}$. The divisibility conditions (\ref{izo}) reduce to the following two conditions (recall that we write $B_i$  for $b_i+b_{i+1}$): 
\begin{align*}
t&\mid -\alpha B_{i_1}+ \delta C_{i_1} - B_{i_1} C_{i_1}t^{-1}\gamma,\\  \nonumber 
t&\mid -\alpha (B_{i_1}+B_{i_2})+\delta(C_{i_1}+C_{i_2})  -(B_{i_1}+B_{i_2})(C_{i_1}+C_{i_2})t^{-1}\gamma. 
\end{align*}
Here, we assume that we started numbering from $i_1$, and that $i_1<i_2<i_3.$

Since  there are no conditions attached to $\beta$, we set it to be 0.  If we set 
\begin{align*}
-\alpha B_{i_1}+ \delta C_{i_1} - B_{i_1} C_{i_1}t^{-1}\gamma &=0,\\
-\alpha (B_{i_1}+B_{i_2})+\delta(C_{i_1}+C_{i_2})  -(B_{i_1}+B_{i_2})(C_{i_1}+C_{i_2})t^{-1}\gamma &=0,
\end{align*}
then we get 
$$\alpha(B_{i_1}+B_{i_2})[C_{i_1}^{-1} (C_{i_1}+C_{i_2}) -1]   +\delta (C_{i_1}+C_{i_2}) [B_{i_1}^{-1} (B_{i_1}+B_{i_2}) -1]=0.$$

If $t\mid C_{i_1}^{-1} (C_{i_1}+C_{i_2}) -1$, then  $t\mid C_{i_2}$, which is not true. It follows that $C_{i_1}^{-1} (C_{i_1}+C_{i_2}) -1$  is invertible. The same holds for $B_{i_1}^{-1} (B_{i_1}+B_{i_2}) -1$. Thus, if we set $\delta=1$, then we get $$\alpha=- (C_{i_1}+C_{i_2})(B_{i_1}+B_{i_2})^{-1} [B_{i_1}^{-1} (B_{i_1}+B_{i_2}) -1] [C_{i_1}^{-1} (C_{i_1}+C_{i_2}) -1]^{-1},$$ and 

$$\gamma=t (-\alpha C_{i_1}^{-1}+B_{i_1}^{-1}).$$
Hence, 
{\renewcommand\arraystretch{2.5}
$$\varphi_0=\begin{pmatrix}- (C_{i_1}+C_{i_2})(B_{i_1}+B_{i_2})^{-1} [B_{i_1}^{-1} (B_{i_1}+B_{i_2}) -1] [C_{i_1}^{-1}(C_{i_1}+C_{i_2}) -1]^{-1} & \,\,\,\,\,\,\,\, 0 \\  t (-\alpha C_{i_1}^{-1}+B_{i_1}^{-1}) & \,\,\,\,\,\,\,\,1  \end{pmatrix}.$$}
The other invertible matrices $\varphi_i$ are now determined from $\varphi_ix_i=x_i\varphi_{i-1}$. 
Note that all of them are invertible because their determinant is equal to $\alpha\delta-\beta\gamma$ which is an 
invertible element.  

If $\{i_1,i_2,i_3\}\neq \{j_1,j_2,j_3\}$, then $\MM$ and $\MM'$ do not satisfy the same divisibility conditions.  It follows by Theorem \ref{izom} that $\MM$ and $\MM'$ are not isomorphic. 
\end{proof}

The previous theorem tells us that the module $\MM_{i_1,i_2,i_3}$ only depends on the divisibility conditions of the coefficients $b_i$, so if we have two different tuples satisfying the same divisibility conditions, then they give rise to isomorphic modules. In total, there are $5\choose{3}$ non-isomorphic indecomposable modules that arise this way, one for each subset of $\{1,2,3,4,5\}$ with three elements.

\subsection{Four of the sums $B_i$ are not divisible by $t$}
Assume that  $t\nmid b_{i_l}+b_{i_l+1}$, $l=1,2,3,4$, and  $t\mid b_{j}+b_{j+1}$,  where $\{i_1, i_2, i_3, i_4\}\cup \{j\}=\{1,3,5,7,9\}$.  Since $\sum_{i=1}^{10}b_i=0$, it follows that $t\mid \sum_{l=1}^4B_{i_l}.$ Recall that we assume that the divisibility conditions from Theorem~\ref{t6}  hold so that the constructed module is indecomposable. Denote this module by $\MM$.  It means that $t\nmid B_{i_l}+B_{i_{l+1}}$ for at least one index $l$. If $t\nmid B_{i_l}+B_{i_{l+1}}$, then $t\nmid B_{i_{l+2}}+B_{i_{l+3}}$. For the remaining two sums $B_{i_{l+1}}+B_{i_{l+2}}$ and $B_{i_{l+3}}+B_{i_{l}}$, either both of them are divisible by $t$ or none of them is. Thus, we have to distinguish between these subcases.

Before we start considering these subcases, we recall that two modules that do not satisfy the same divisibility conditions are not isomorphic. Let $(c_i)$ be another tuple giving rise to another indecomposable module $\MM'$. Here, we assume that  $t\nmid c_{j_l}+c_{j_l+1}$, $l=1,2,3,4$, and  $t\mid c_{i}+c_{i+1}$,  where $\{j_1, j_2, j_3, j_4\}\cup \{i\}=\{1,3,5,7,9\}$.  Since $\sum_{i=1}^{10}c_i=0$, it follows that $t\mid \sum_{l=1}^4C_{j_l}.$ Also, we assume that  $t\nmid C_{j_l}+C_{j_{l+1}}$ for at least one index $l$.

Now we examine if the modules $\MM$ and $\MM'$ are isomorphic when they satisfy the same divisibility conditions. Let $\{i_1,i_2,i_3,i_4\}=\{j_1,j_2,j_3,j_4\}$.  Here, we can assume that these odd numbers are consecutive.

The first subcase is when two of the sums $B_{i_j}+B_{i_j+1}$  are divisible by $t$, and two are not. Thus, we assume that $t\nmid B_{i_1}+B_{i_2}$, $t\nmid B_{i_3}+B_{i_4}$, $t\mid B_{i_2}+B_{i_3}$,  and $t\mid B_{i_4}+B_{i_1}$.  The same conditions hold for $\MM'$, so $t\nmid C_{i_1}+C_{i_2}$, $t\nmid C_{i_3}+C_{i_4}$, $t\mid C_{i_2}+C_{i_3}$,  and $t\mid C_{i_4}+C_{i_1}$.  Note that if $t\nmid B_{i_1}+B_{i_2}$, then $t\nmid B_{i_3}+B_{i_4}$   because $t\mid \sum_{l=1}^4B_{i_l}.$ Analogously,  if $t\mid B_{i_2}+B_{i_3}$, then $t\mid B_{i_4}+B_{i_1}$.

\begin{theorem} \label{prop:two-sums-not-matching}
If $\MM$ and $\MM'$ are such that $t\nmid B_{i_1}+B_{i_2}$, $t\mid B_{i_2}+B_{i_3}$, $t\nmid C_{i_1}+C_{i_2}$, and $t\mid C_{i_2}+C_{i_3}$, then $\MM$ and $\MM'$ are isomorphic.
\end{theorem}

\begin{proof} Keeping the same notation as before when constructing isomorphisms, the divisibility conditions (\ref{izo}) reduce to: 
\begin{align*}
t&\mid -\alpha B_{i_1}+ \delta C_{i_1} - B_{i_1}C_{i_1}t^{-1}\gamma,\\
t&\mid -\alpha (B_{i_1}+B_{i_2})+\delta (C_{i_1}+C_{i_2})  -(B_{i_1}+B_{i_2})(C_{i_1}+C_{i_2})t^{-1}\gamma ,
\end{align*}
because $t\mid B_{i_2}+B_{i_3}$, $t\mid C_{i_2}+C_{i_3}$, $t\mid \sum_{l=1}^4B_{i_l}$, and $t\mid \sum_{l=1}^4C_{i_l}.$ Now, we proceed as in the proof of Theorem \ref{M123} in order to construct an isomorphism between $\MM$ and $\MM'$. 
\end{proof}

In total, this subcase gives  $2 {5 \choose 4 }$ non-isomorphic indecomposable modules. There are two modules for every choice of a four-element subset of $\{1,3,5,7,9\}$. 

The second subcase is when none of the sums $B_{i_1}+B_{i_2}$  is divisible by $t$. Thus, we assume that $t\nmid B_{i_l}+B_{i_{l+1}}$ and  $t\nmid C_{i_l}+C_{i_l+1}$, for $l=1,2,3,4$.  

\begin{theorem}  \label{inf1}
If $t\nmid B_{i_l}+B_{i_l+1}$ and $t\nmid C_{i_l}+C_{i_l+1}$, for $l=1,2,3,4,$ then the modules $\MM$ 
and $\MM'$ are isomorphic if and only if 
$$
t\mid B_{i_1} C_{i_2} B_{i_3} C_{i_4} - C_{i_1}B_{i_2} C_{i_3}B_{i_4}.
$$
\end{theorem}
\begin{proof}
As before, if there were an isomorphism between $\MM$ and $\MM'$, its coefficients would have to satisfy the 
following conditions that we obtain from (\ref{izo}): 
\begin{align*}
t&\mid -\alpha B_{i_1}+ \delta C_{i_1} - B_{i_1}C_{i_1}t^{-1}\gamma,\\
t&\mid -\alpha (B_{i_1}+B_{i_2})+\delta (C_{i_1}+C_{i_2})  -(B_{i_1}+B_{i_2})(C_{i_1}+C_{i_2})t^{-1}\gamma ,\\
t&\mid \alpha B_{i_4}-\delta C_{i_4} -B_{i_4}C_{i_4}t^{-1}\gamma,
\end{align*}
because $t\mid \sum_{l=1}^4B_{i_l}$, and $t\mid \sum_{l=1}^4C_{i_l}.$

From these we get that 
\begin{align*}
t&\mid \alpha C_{i_2}[C_{i_1}(C_{i_1}+C_{i_2})]^{-1}-\delta B_{i_2}[B_{i_1}(B_{i_1}+B_{i_2})]^{-1},\\
t&\mid \alpha C_{i_3}[B_{i_4}(C_{i_1}+C_{i_2})]^{-1}-\delta  B_{i_3}[B_{i_4}(B_{i_1}+B_{i_2})]^{-1}.
\end{align*}
Finally, from the last two relations we get 
$$t\mid \alpha [B_{i_1} C_{i_2} B_{i_3} C_{i_4} - C_{i_1}B_{i_2} C_{i_3}B_{i_4}].$$
If $t\nmid B_{i_1} C_{i_2} B_{i_3} C_{i_4} - C_{i_1}B_{i_2} C_{i_3}B_{i_4}$, then there is no isomorphism between $\MM'$ and $\MM$. If $t\mid B_{i_1} C_{i_2} B_{i_3} C_{i_4} - C_{i_1}B_{i_2} C_{i_3}B_{i_4},$ then we simply set $\alpha=1$, and compute $\delta$ and $\gamma$ from the above relations (as before, we set $\beta=0$).
\end{proof}

\begin{rem}\label{classify} To classify all non-isomorphic indecomposable modules given by the previous theorem, we use exactly the same arguments as in Section 5 in \cite{BBL}. To each $\beta \in \mathbb C\setminus\{-1,0,1\}$ corresponds an indecomposable module $M_{\beta}$ defined by $B_{i_1}=1$, $B_{i_2}=\beta$, $B_{i_3}=-1$, $B_{i_4}=-\beta$, and $B_{i_5}=0$. Here, $i_j<i_{j+1}$. It was proved in \cite{BBL} that $M_{\beta}\cong M_{\gamma}$ if and only if $\beta=\pm \gamma$, and that for a given indecomposable module $\MM$ there exists $\beta$ such that $\MM\cong \MM_{\beta}$. This means that all indecomposable modules in this case are parameterized by a single parameter $\beta$. 
Obviously, there are five different families (each in bijection with $\mathbb C$), depending on which $B_{i_j}$ is set to be divisible by $t$. 
\end{rem}

\subsection{None of the five sums $B_i$ is divisible by $t$} 
Since $t\nmid B_{i}+B_{i+2}$ for at least one odd index $i$, there are  three subcases we have to consider. The first subcase is when $t\mid B_i+B_{i+2}$ and $t\mid B_{i+2}+B_{i+4}$ for a unique odd index $i$. The second subcase is when $t\mid B_i+B_{i+2}$ for a unique odd index $i$. The third subcase is when $t\nmid B_i+B_{i+2}$ for all odd $i$. 
In the last two  cases, we get infinitely many non-isomorphic indecomposable modules as we will show.

As before, let us assume that $(c_i)$ is another $10$-tuple giving rise to a module $\MM'$ satisfying  the same divisibility conditions as the module $\MM$. 

Assume that there is a unique odd index $i$ such that $t\mid B_i+B_{i+2}$ and $t\mid B_{i+2}+B_{i+4}$. Recall that whenever we state the divisibility conditions for the $b_i$'s, we assume that the same conditions hold for the $c_i$'s. 
\begin{theorem}
If $l$ is odd such that $t\mid B_{l}+B_{l+2}$, $t\mid B_{l+2}+B_{l+4}$,  $t\nmid B_{i}+B_{i+2}$, for $i\neq l,l+2$, and $t\mid C_{l}+C_{l+2}$, $t\mid C_{l+2}+C_{l+4}$,  $t\nmid C_{i}+C_{i+2}$, for $i\neq l,l+2$, 
then $\MM$ and $\MM'$ are isomorphic. 
\end{theorem}
\begin{proof}
Without loss of generality we can assume that $l=3$. Keeping the same notation as before when constructing isomorphisms, because $t\mid B_{3}+B_{5}$, $t\mid B_{5}+B_{7}$,  the divisibility conditions (\ref{izo}) reduce to: 
\begin{align*}
t&\mid -\alpha B_{1}+ \delta C_{1} - B_{1}C_{1}t^{-1}\gamma,\\
t&\mid -\alpha (B_{1}+B_{3})+\delta (C_{1}+C_{3})  -(B_{1}+B_{3})(C_{1}+C_{3})t^{-1}\gamma.
\end{align*}
Now, we proceed as in the proof of Theorem \ref{M123} in order to construct an isomorphism between $\MM$ and $\MM'$. 
\end{proof}
There are five non-isomorphic indecomposable modules arising in this subcase, one for each index $l\in \{1,3,5,7,9\}$.
\vspace{5mm}

Assume that there is a unique odd index $i$ such that $t\mid B_i+B_{i+2}$ and $t\nmid B_{j}+B_{j+2}$, for $j\neq i$. 
\begin{theorem}
If $l$ is odd such that $t\mid B_{l}+B_{l+2}$, $t\nmid B_{i}+B_{i+2}$,   for $i\neq l$, and $t\mid C_{l}+C_{l+2}$, $t\nmid C_{i}+C_{i+2}$,   for $i\neq l$, 
then $\MM$ and $\MM'$ are isomorphic if and only if 
$$
t\mid (B_{l-2}+B_{l} )C_{l} B_{l+4} C_{l+6} - (C_{l-2}+C_l)B_{l} C_{l+4}B_{l+6}.
$$
\end{theorem}
\begin{proof}
Without loss of generality, assume that $l=3$. 
As before, if there were an isomorphism between $\MM$ and $\MM'$, its coefficients would have to satisfy the 
following conditions that we obtain from (\ref{izo}): 
\begin{align*}
t&\mid -\alpha B_{1}+ \delta C_{1} - B_{1}C_{1}t^{-1}\gamma,\\
t&\mid -\alpha (B_{1}+B_{3})+\delta (C_{1}+C_{3})  -(B_{1}+B_{3})(C_{1}+C_{3})t^{-1}\gamma ,\\
t&\mid \alpha B_{9}-\delta C_{9} -B_{9}C_{9}t^{-1}\gamma.
\end{align*}
Now we proceed as in the proof of Theorem \ref{inf1}. 
\end{proof}
Let us parameterize the indecomposable modules from the previous theorem. Let $\beta \in \mathbb C$ and denote by $\MM_{\beta}$ the indecomposable module whose coefficients $b_i$ satisfy $B_1=\beta$, $B_3=1$, $B_5=-1$, $B_7=-\beta-1$, and $B_9=1$. Also, $\beta\neq 0,-1,-2.$  

\begin{prop} Let $\MM'$ be a module such that $t\mid C_{l}+C_{l+2}$, $t\nmid C_{i}+C_{i+2}$,   for $i\neq l$. There exists $\beta\in \mathbb C\setminus \{0,-1,-2\}$ such that $\MM'\cong \MM_{\beta}.$ 
\end{prop}
\begin{proof} Assume again that $l=3$. By the previous theorem, if $\MM$ and $\MM_{\beta}$ were isomorphic, then $t\mid (C_1+C_3)C_7\beta+C_3C_9 (\beta+1)^2$. If $\gamma_i$ is the constant term of $C_i$, then we set $\beta$ to be a solution of the equation $$(\beta+1)^2=-\gamma_3^{-1}\gamma_9^{-1}(\gamma_1+\gamma_3)\gamma_7.$$   
Since the right-hand side of the previous equation is invertible, $\beta\neq -1.$  If $\beta=0$ or $\beta=-2$, then $-\gamma_3\gamma_9=(\gamma_1+\gamma_3)\gamma_7,$ and subsequently,   $-\gamma_3\gamma_9-\gamma_3\gamma_7=\gamma_1\gamma_7.$ From $\gamma_1+\gamma_7+\gamma_9=0$ (this follows from $t\mid C_1+C_7+C_9$), we get $\gamma_3\gamma_1=\gamma_1\gamma_7$. Thus, $\gamma_3=\gamma_7$. This means that $\gamma_7+\gamma_5=\gamma_3+\gamma_5=0$, which is not possible since $t\nmid C_5+C_7$.  Hence, $\beta\neq 0, -1,-2$.  
\end{proof}

It is clear that $\MM_{\beta}\cong \MM_{\gamma}$ if and only if $(1+\beta)^2=(1+\gamma)^2$. This means that either $\beta=\gamma$ or $\beta +\gamma =-2.$ This means that the non-isomorphic indecomposable modules given in this subcase are parameterized by the set $\mathbb C\setminus \{0,-1,-2\}$, where we identify two points if they sum up to $-2$.

There are five different families (each in bijection with $\mathbb C$) of indecomposable modules arising in this subcase, one for each $l\in \{1,3,5,7,9\}.$
\vspace{5mm}

Assume that  $t\nmid B_{i}+B_{i+2}$, for all odd $i$. 
\begin{theorem}  \label{worst}
If $t\nmid B_{i}+B_{i+2}$ and $t\nmid C_{i}+C_{i+2}$, for all odd $i$, then the modules $\MM$ 
and $\MM'$ are isomorphic if and only if the following conditions hold: 
\begin{align*}
t&\mid C_1B_3(C_5+C_7)B_9-B_1C_3(B_5+B_7)C_9,\\
t&\mid   C_1B_3C_5(B_7+B_9)-B_1C_3B_5(C_7+C_9).
\end{align*}
\end{theorem}
\begin{proof}
As before, if there were an isomorphism $\varphi=(\varphi_i)$ between $\MM$ and $\MM'$, the coefficients of $\varphi_0=\begin{pmatrix}\alpha & \beta \\ \gamma & \delta  \end{pmatrix}
$ would have to satisfy the 
following conditions that we obtain from (\ref{izo}): 
\begin{align*}\label{izo} 
t&\mid -\alpha B_{1}+ \delta C_{1} - B_{1}C_{1}t^{-1}\gamma,\\
t&\mid -\alpha (B_{1}+B_{3})+\delta (C_{1}+C_{3})  -(B_{1}+B_{3})(C_{1}+C_{3})t^{-1}\gamma ,\\
t&\mid \alpha (B_{7}+B_{9})-\delta (C_{7}+C_{9})  -(B_{7}+B_{9})(C_{7}+C_{9})t^{-1}\gamma ,\\
t&\mid \alpha B_{9}-\delta C_{9} -B_{9}C_{9}t^{-1}\gamma. 
\end{align*}
Now, we use the same calculations as in the proof of Theorem \ref{inf1} in order to obtain the desired divisibility conditions.    The trick is to use any three of the above divisibility conditions and treat them as in the proof of Theorem \ref{inf1}. For example, we use the first two and the last condition, and treat $B_5+B_7$ as $B_5$ in the proof of Theorem \ref{inf1}. This gives us that $t\mid C_1B_3(C_5+C_7)B_9-B_1C_3(B_5+B_7)C_9$. Analogously, use the first three conditions and treat $B_7+B_9$ as $B_7$ in the proof of Theorem \ref{inf1} to obtain $t\mid   C_1B_3C_5(B_7+B_9)-B_1C_3B_5(C_7+C_9).$

Conversely, if the given conditions hold, by setting $\alpha=1$, one easily computes $\delta$ and $\gamma$ from the above relations: $\delta=B_1C_3B_3^{-1}C_1^{-1}(B_1+B_3)(C_1+C_3)^{-1}$, $\gamma=t (-C_1^{-1}+\delta B_1^{-1})$. We set $\beta=0.$
\end{proof}

We are left to parameterize the indecomposable modules from the previous theorem.

Denote by $\MM$ the indecomposable module corresponding to the coefficients $B_i=b_{i}+b_{i+1}$, for odd $i$. Since $\sum_i B_i=0$, we can rescale so that one of the $B_i$'s is equal to 1, say $B_7$, because from the previous theorem it holds that the module $\MM$ is isomorphic to the module corresponding to the coefficients $B_i^{'}=B_iB^{-1}_{7}$, for $i\neq 7$, and $B_7^{'}=1$. 

Let $\MM'$ denote another indecomposable module determined by the coefficients $C_1=\alpha$, $C_3=\beta$, $C_5=\gamma$, $C_7=1$, and $C_9=\delta$, all of them being complex numbers such that   $\alpha+\beta+\gamma+1+\delta=0$. Also, $\alpha,\beta,\gamma, \delta \neq 0$, $\gamma,\delta\neq -1$, $\alpha+\beta\neq 0$, $\alpha+\delta\neq 0$, $\gamma+\beta\neq 0$. Under the assumption that $\MM$ and $\MM'$ are isomorphic, we will express $\alpha,\beta,$ and $\delta$ as a function of $\gamma$ and coefficients $B_i$. This will help us to find an appropriate parameterization of indecomposable modules in this subcase. 

By the previous theorem, it must hold 
\begin{align*}
t&\mid \alpha (1+\gamma)B_3B_9-\beta\delta B_1(B_5+1),\\
t&\mid   \alpha\gamma B_3(1+B_9)-\beta(1+\delta)B_1B_5.
\end{align*}
These two relations imply that 
\begin{align*}
t&\mid (1+\delta)(1+\gamma)B_5B_9-\gamma\delta (1+B_9)(B_5+1),\\
t&\mid   B_1^{-1}B_5^{-1}[\alpha\gamma B_3(1+B_9)-\alpha(1+\gamma)B_3B_5B_9(1+B_5)^{-1}]-\beta.
\end{align*}
The first of the last two relations is equivalent to 
\begin{align*}
t&\mid B_5B_9(\alpha+\beta)-\gamma\delta (B_1+B_3).
\end{align*}
Since, $\alpha+\beta=-1-\gamma-\delta$, this yields
\begin{align*}
t&\mid -(1+\gamma)[1+\gamma B_5^{-1}B_9^{-1}(B_1+B_3)]^{-1}-\delta.
\end{align*}
The last divisibility condition is under the assumption that $1+\gamma B_5^{-1}B_9^{-1}(B_1+B_3)$ is invertible, i.e., that $t\nmid 1+ \gamma B_5B_9(B_1+B_3)^{-1}$. If this condition holds, then $\delta\neq 0$. If  $B_5^{-1}B_9^{-1}(B_1+B_3)=1$, then from $B_1+B_3+B_5+1+B_9=0$ we get $(B_5+1)(B_9+1)=0$, which is not possible. Thus, $1+\gamma B_5^{-1}B_9^{-1}(B_1+B_3)\neq 1+\gamma$ and $\delta\neq -1$.  Also, from $t\mid B_5B_9(\alpha+\beta)-\gamma\delta (B_1+B_3)$ follows that $\alpha+\beta\neq 0$ because $\gamma\delta (B_1+B_3)$ is invertible.

 From $t\mid B_5B_9(\alpha+\beta)-\gamma\delta (B_1+B_3)$  and $t\mid   B_1^{-1}B_5^{-1}[\alpha\gamma B_3(1+B_9)-\alpha(1+\gamma)B_3B_5B_9(1+B_5)^{-1}]-\beta,$ we get 
 $$t\mid \alpha [(1-B_1^{-1}(1+B_5)^{-1}B_3B_9)-\gamma B_1^{-1}(1+B_5)^{-1}B_3B_5^{-1}(B_1+B_3))]-\delta\gamma B_5^{-1}B_9^{-1}(B_1+B_3).$$
If $t\nmid B_5B_3^{-1}(B_1+B_3)^{-1}(B_1(1+B_5)-B_3B_9)-\gamma$, then  $(1-B_1^{-1}(1+B_5)^{-1}B_3B_9)-\gamma B_1^{-1}(1+B_5)^{-1}B_3B_5^{-1}(B_1+B_3))$ is invertible, and so $\alpha\neq 0$.
If $t\mid \alpha +\delta$, then from $t\mid \alpha [(1-B_1^{-1}(1+B_5)^{-1}B_3B_9)-\gamma B_1^{-1}(1+B_5)^{-1}B_3B_5^{-1}(B_1+B_3))]-\delta\gamma B_5^{-1}B_9^{-1}(B_1+B_3)$ direct computation yields that $t\mid 1+ \gamma B_5B_9(B_1+B_3)^{-1}$ which we already assumed is not true. Hence, $t\nmid \alpha+\delta$ and $\alpha+\delta\neq 0.$ It is shown in a similar fashion that $\beta+\gamma\neq 0$. 

Therefore, if we define $\delta$ to be the constant term of $-(1+\gamma)[1+\gamma B_5^{-1}B_9^{-1}(B_1+B_3)]^{-1}$, $\alpha$ to be the constant term of  $-\delta\gamma [(1-B_1^{-1}(1+B_5)^{-1}B_3B_9)-\gamma B_1^{-1}(1+B_5)^{-1}B_3B_5^{-1}(B_1+B_3))]^{-1}B_5^{-1}B_9^{-1}(B_1+B_3)$, and $\beta$ to be the constant term of $-\alpha B_1^{-1}(1+B_5)^{-1}B_3B_9[1+\gamma B_5^{-1}B_9^{-1}(B_1+B_3)]$, we get a parameterization of the coefficients of $\MM$ with only one complex parameter $\gamma$ involved. The other parameters, $\alpha$, $\beta$, and $\delta$, are expressed as a function of $\gamma$ and the coefficients $B_i$. If we want to fix a value of one of the parameters $\alpha$, $\beta$, $\gamma$, and $\delta$, then the sum of the remaining three  is fixed. Thus, one of them is determined by the remaining two, so we end up with a parameterization of the form, e.g., $\alpha,-2-\alpha-\gamma, \gamma,1,1$, with two parameters. Two such modules corresponding to different $5$-tuples of parameters are isomorphic if and only if the divisibility conditions from Theorem~\ref{worst} are satisfied. Thus, we identify two $5$-tuples if and only if they satisfy the divisibility conditions from Theorem~\ref{worst}.         

 In this subcase there is only one family of indecomposable modules.

\bibliographystyle{abbrv}
\bibliography{biblio}

\end{document}